\documentclass{birkjour}
\usepackage[english]{babel}
\usepackage[utf8]{inputenc}
\usepackage{amsmath}
\usepackage{graphicx}
\usepackage[colorinlistoftodos]{todonotes}
\usepackage{amsthm}
\usepackage{amsfonts}
\usepackage{hyperref}
\usepackage{upgreek}

\newtheorem{theorem}{Theorem}[subsection]
\newtheorem*{theorem*}{Theorem}
\newtheorem{definition}[theorem]{Definition}

\newtheorem{lemma}[theorem]{Lemma}

\title{Convergence of  Taylor Series of real normed division algebras.}

\author{Eric Dolores}

\date{\today}

\begin{document}
\begin{abstract}We propose a new way to compute the radius of convergence for Quaternionic hyperholomorphic functions  and  for Octonion analytic functions.
 We extend the theorem of Cauchy-Hadamard and the theorem of Abel on convergence of series to Quaternionic  analysis with the Fueter Operator,  to Quaternionic  analysis with the Moisil–Th\'eodoresco operator and to the Octonion analysis with the Fueter operator. 
\end{abstract}

\maketitle

\section{Background}
William Rowan Hamilton wanted to define  an operation on $R^3$ with  analogous properties to  multiplication of complex numbers. On October 16, 1846, he found the rules to define an associative product on $R^4$  with  unit so that every non-zero element has an inverse. This product is not commutative. Hamilton called this structure on $R^4$  the Quaternion numbers which we denote as \[\mathbb{H}=\{x_0+ix_1+jx_j+kx_k|(x_0,x_1,x_2,x_3)\in R^4, i^2=j^2=k^2 = -1, ijk=-1\}.\] Quaternions are closely related to rotations on $R^3$ and $R^4$. 

 Dr. Fueter proposed the study of Quaternionic functions that are zeroes of the operator: 

\[D=\frac{\partial}{\partial x_0}+i\frac{\partial }{\partial x_1}+j\frac{\partial }{\partial x_2}+k\frac{\partial }{\partial x_3}.\]
which generalizes the Cauchy Riemann operator in complex analysis \cite{Fueter}. We call this functions hyperholomorphic.  We will study the radius of convergence of power series expansions of hyperholomorphic functions. 

\begin{definition}
The Fueter's basis is given by the hyperholomorphic functions $\zeta_1(x)=x_1-ix_0, \,\zeta_2(x)=x_2-jx_0,\, \zeta_3(x)=x_3-kx_0.$
Where $x=x_0+ix_1+jx_2+kx_3.$ They form the basis of hyperholomorphic polynomials of degree one.
\end{definition}

We write the Taylor expansion of a hyperholomorphic function using  Fueter's Basis:

\begin{eqnarray*}
f(x)&=&\sum_0^\infty \sum_{\stackrel{\nu=(n_1,n_2,n_3)}{n_1+n_2+n_3=n}} P_\nu a_\nu,\end{eqnarray*}
 where $a_\nu\in\mathbb{H} $ for all $n\in N$ and $\nu=(n_1,n_2,n_3), n_1+n_2+n_3=n$   and $P_\nu$ are degree $n$ hyperholomorphic homogeneus polynomials. The polynomials $P_\nu$  are permutational products of the Fueter's basis. For example $2*P_{(1,1,0)}= \zeta_1(x)\zeta_2(x)+\zeta_2(x)\zeta_1(x).$  We will give the full definition on the next section.

On the current literature, such as [\cite{Holo2014}, page 168], the radius of convergence of a hyperholomorphic function is calculated by expanding $\|f(x)\|$ as sum of monomial terms: 
 \begin{eqnarray*}
\|f(x)\|&\leq&\sum_0^\infty \sum_{\stackrel{\nu=(n_1,n_2,n_3)}{n_1+n_2+n_3=n}} \|P_\nu\|\| a_\nu\|,\\
&\leq&\sum_{n=0}^\infty \sum_{v_0+v_1+v_2+v_3=n}  |x_0^{v_0}x_1^{v_1}x_2^{v_2}x_3^{v_3}| \|a_{(v_0,v_1,v_2,v_3)}\|,\end{eqnarray*}
Here the coefficients $a_{(v_0,v_1,v_2,v_3)}$ are a linear combination of the coefficients $\{a_\nu\}_{\|\nu\|=n}$. Then we apply the formula \[ \rho = \lim_{n \rightarrow \infty} (\sum_{\|(v_0,v_1,v_2,v_3)\|=n}{\|a_{(v_0,v_1,v_2,v_3)}}\|^{1/n})^{-1}\] to guarantee that $f$  converges on $\{x\in\mathbb{H}|\,\,\|x\|<\rho\}$.

 In this paper we calculate a different radius of convergence $\rho^\prime.$ We use the Taylor expansion of the function in terms of the Fueter basis. We introduce an auxiliary  norm $\|\cdot\|^\prime$ that determines the region of convergence of the function $f$.
 
 The main idea here is that since we expand the homogenous polynomials in the Fueter basis, we also use a norm related to that basis. This norm help us understand hyperholomorphic functions better. To support this idea we use the norm $\|\cdot\|^\prime$  to extend Abel, and Cauchy-Hadamard theorems on convergence of complex series to the Quaternoinic case. 

 The non commutativity of Quaternions set the shape of the Taylor expansions.   We will used the fact that the norm $\|\cdot\|$ is multiplicative to bound the Taylor expansion in terms of the norm $\|\cdot\|^\prime$. We will never worry about commutativity since we will work all the time with real numbers. 
 The same ideas are applied to analytic functions on Octonions. The reason is that Octonions and Quaternions are finite real division algebras that accept a  multiplicative norm.  

 Finally we also show equivalent theorems for the zeroes of the Quaternionic Moisil-Theodoresco operator. 
 
 As far as the author is aware, this is the first time those theorems are extended to Quaternions and Octonions. A weak version of Abel lemma, which is closely related, has been proved in [\cite{Malonek1990}, equation (18)] using estimations in terms of $\|\zeta_i\|$.

The paper is organized as follows. In section \ref{Intro} we describe the auxiliary norm $\|\cdot\|^\prime$ and  its general properties.
 In section  \ref{Radio} we introduce the radius of convergence $\rho^\prime$ and we prove the Quaternionic version of Abel and Cauchy-Hadamard theorems.
In section \ref{Comparison} we study relationships of our radius of convergence with other definitions. 
In section \ref{O} we extend the theorems to Octonion analytic functions.  In section \ref{MT} we study the case of the Mosil Teodoresco operator and the Octonions. In the final section we discuss open questions.

\section{Introduction}
\label{Intro}
In this section we describe the Taylor series expansion of a hyperholomorphic function. We introduce a norm $\|\cdot\|^\prime$; we use the norm to find a bound on the homogeneous term of degree $n$ of the Taylor series expansion. In the next section we will use this bound to describe the radius of convergence.

 \begin{definition}
A function $f:\mathbb{H}\rightarrow\mathbb{H}$ is said to be (left) hyperholomorphic in a neighborhood $V$ of the origin, if $f$ is real differentiable on $V$ and if $Df=0$ when $D$  is the Cauchy-Riemann-Fueter operator \[D=\frac{\partial}{\partial x_0}+i\frac{\partial }{\partial x_1}+j\frac{\partial }{\partial x_2}+k\frac{\partial }{\partial x_3}.\]
\end{definition}

We study the following Taylor expansion of a (left) hyperholomorphic function $f$ at the origin given in terms of non commutative polynomials

\begin{eqnarray*}
f(x)&=&\sum_0^\infty \sum_{\stackrel{\nu=(n_1,n_2,n_3)}{n_1+n_2+n_3=n}} P_\nu a_\nu,\end{eqnarray*}
 where 
\[P_\nu=\frac{1}{n!}\sum_{(i_1,\cdots,i_n)\in A_\nu} \zeta_{i_1}\cdots \zeta_{i_n},\] and the sum is over $\vec{\nu}\in A_\nu$, $A_\nu$ is the set of all possible words with  $n_1$ letters $1$, $ n_2$ letters $2$ and $n_3$ letters $3$, see \cite{Malonek1990}.

We propose the use of the following auxiliary norm to analyse the Taylor expansion of $f$.

\begin{definition}
We denote by $\|\,\cdot\,\|^\prime:\mathbb{H}\rightarrow \mathbb{R}$ the norm 
\[||x_0+ix_1+jx_2+kx_3||^\prime=\max\{||x_0+ix_1||,||x_0+jx_2||,||x_0+kx_3||\},\]
where $||r+ls||=\sqrt{r^2+s^2}, l\in\{i,j,k\}$  is the euclidean distance.
\end{definition}

 The corresponding regions  \[B(0,r):=\{x\in \mathbb{H}|\,(\sum_0^3 x_i^2)^{1/2}<r\},  B^\prime(0,r):=\{x\in \mathbb{H}|\,||x||^\prime<r\}\] satisfy  \[B(0,r)\subset B^\prime(0,r).\] 
 
For example $.9i+.9j\in B^\prime(0,1)$ but $.9i+.9j\not\in B(0,1).$
The shape of  $B^\prime(0,r)$ is a 4 dimensional object whose 3-dim boundary has cubes and poly-cylinders.
For purely imaginary values, we get \[ ||x_1i+x_2j+x_3k ||^\prime=max\{{|x_1|,|x_2|,|x_3|}\}\] while for complex numbers, we get $ ||x_0+x_1i ||^\prime=\sqrt{x_0^2+x_1^2}$.

We compute
\begin{eqnarray}
\hbox{ volume }(B^\prime(0,r))&=&16\int^r_{t=0}(r^2-t^2)^{3/2}dt\\
&=&6/\pi*\hbox{ volume }(B(0,r))
\end{eqnarray}

From the following examples: \begin{eqnarray*}
||i(1+2j)||^\prime&=&
||i+2k||^\prime\\
&=&2\\
&<&||i||^\prime||1+2j||^\prime\\
&=&\sqrt{5}\\
&=&||1+2j||^\prime\\
&=&||i^{-1}(i+2k)||^\prime\\
&>&||i^{-1}||^\prime||1+2k||^\prime\\
&=&2
\end{eqnarray*}

we conclude that in general
$||xy ||^\prime$ cannot be compared with $||x ||^\prime||y ||^\prime$ as $||i(1+2j)||^\prime<||i||^\prime||1+2j||^\prime$ and $||i^{-1}(i+2k)||^\prime  > ||i^{-1}||^\prime||1+2k||^\prime$. This won't affect our calculations because the usual norm $||\cdot||$ is multiplicative, and we will only use  the  norm $\|\cdot\|^\prime$ to quantify terms on the homogeneous polynomials $P_\nu$ as we explain below.

The usual method to compute the radius of convergence considers   
\[\|\zeta_{i_1}\cdots \zeta_{i_n}\|\leq  (\|x_{i_1}\|+\|x_{i_0}\|)\cdots (\|x_{i_1}\|+\|x_{i_0}\|)\leq 2^n \|x\|^n,\] by using that $\|x_{i_0}\|<\|x\|$, instead  we skip the triangle inequality by considering

\begin{eqnarray}
\|\zeta_{i_1}\cdots \zeta_{i_n}\|&=&\|\zeta_{i_1}\|\cdots \|\zeta_{i_n}\|\\
&\leq& (\|x\|^\prime)^n.\label{Bound}
\end{eqnarray}

Going from the Fueter variables to the real variables always adds a $2^n$ coefficient. Now we will use this norm $\|\cdot\|^\prime$ to bound the value of $\|f(x)\|$.

Given $\nu=(n_1,n_2,n_3)$ we define $A_\nu$ as the set of all possible words $\vec{\nu}$ with $n_1$ letters $1$, $n_2$ letters $2$, and $n_3$ letters $3$. To each word $\vec{\nu}\in A_\nu$ we consider a product of the Fueter basis where the letters of  $\vec{\nu}$ indicate the position of the  corresponding function  $\zeta_1,\zeta_2,\zeta_3$.
 $ A_\nu$ keeps track of the terms in $P_\nu$.

\begin{definition}
Given the series

$$f(x)=\sum_{n=0}^\infty \sum_{\stackrel{\nu=(n_1,n_2,n_3)}{n_1+n_2+n_3=n}} P_\nu a_\nu,$$
 We define $N(f)$ by 
\[N(f)(x)=\lim_{N\rightarrow\infty} \sum_{n=0}^N \sum_{\stackrel{\nu=(n_1,n_2,n_3)}{n_1+n_2+n_3=n}} ||\frac{a_\nu}{n!}|| \sum_{\stackrel{\vec{\nu}=(i_1,\cdots, i_n)}{\vec{\nu}\in A_{\nu}}} || \zeta_{i_1}\cdots \zeta_{i_n}(x)||.\]
 \end{definition}

 In the remaining of the section we will estimate upper bounds for $N(\sum_{\stackrel{\nu=(n_1,n_2,n_3)}{n_1+n_2+n_3=n}} P_{\nu}a_\nu ).$

Since  $A_{\nu}$, has  $\binom{n}{n_1\,n_2\,n_3}$ possible words we have

\begin{eqnarray}
  N(n!P_{(n_1,n_2,n_3)})(x)&=&\sum_{(i_1,\cdots,i_n)\in A_\nu} || \zeta_{i_1}\cdots \zeta_{i_n}(x)||\\ 
  &\leq&    \sum_{(i_1,\cdots,i_n)\in A_\nu} 
(||x||^{\prime })^{n_1+n_2+n_3}\\ 
  &=&   (||x||^{\prime })^{n_1+n_2+n_3} \binom{n_1+n_2+n_3}{n_1,n_2,n_3},\label{Bound2}
\end{eqnarray}
the inequality comes from \eqref{Bound}.

With help of the norm $\|\cdot\|^\prime$ we make the estimation:

\begin{eqnarray}
  N(\sum_{\stackrel{\nu=(n_1,n_2,n_3)}{n_1+n_2+n_3=n}} P_{\nu}a_\nu )(x)&\leq& \frac{(||x||^{\prime })^{n}}{n!}\sum_{\stackrel{\nu=(n_1,n_2,n_3)}{n_1+n_2+n_3=n}}   \binom{n}{n_1,n_2,n_3} \|a_\nu\|.\label{Bound3}
\end{eqnarray}

The main point of the theorems in  the paper is to show different conditions to bound the right side of expression \eqref{Bound3}.

\section{Convergence Theorems}\label{Radio}
In this section we will show how $\|\cdot\|^\prime$ can be used to describe the radius of convergence of a  Taylor expansion of a hyperholomorphic function. 
This section includes the proof of the main theorems.

We need an auxiliary lemma to guarantee compactly convergence of the Taylor Series.

\begin{lemma} \label{Aux} Let $f(x)=\sum_0^\infty \sum_{\stackrel{\nu=(n_1,n_2,n_3)}{n_1+n_2+n_3=n}} P_\nu a_\nu.$
 If $ N(f)(x)<\infty$ then $f$ converges compactly on
\[\left\{x|\,\,\,||\zeta_1(x)||\leq||\zeta_1(h)||,||\zeta_2(x)||\leq||\zeta_2(h)||,||\zeta_3(x)||\leq||\zeta_3(h)|| \right\}.\]
 \end{lemma}

\begin{proof}
It follows from  Weiestrass $M$-test, see [\cite{Malonek1990}, Theorem 3]. 
\end{proof}

On the next theorems we describe conditions to bound the right side of:

\begin{eqnarray}
  N(\sum_{\stackrel{\nu=(n_1,n_2,n_3)}{n_1+n_2+n_3=n}} P_{\nu}a_\nu )(x)&\leq& \frac{(||x||^{\prime })^{n}}{n!}\sum_{\stackrel{\nu=(n_1,n_2,n_3)}{n_1+n_2+n_3=n}}   \binom{n}{n_1,n_2,n_3} \|a_\nu\|.\label{Bound4}
\end{eqnarray}

\begin{theorem}\label{Abel} (Quaternionic  Abel Theorem) Suppose that there are cons\-tants $r_0, M\in \mathbb{R}, N_0\in \mathbb{R},$ such that for all $n>N_0$ and multi indexes $\nu$ with $||\nu||=n$ we have the bound $||a_\nu||r^n_0\leq M$. Under this hypothesis  the series $f(x)=\sum_0^\infty \sum_{\stackrel{\nu=(n_1,n_2,n_3)}{n_1+n_2+n_3=n}} P_\nu a_\nu
$ converges compactly on $B^\prime(0,r_0).$
\end{theorem}
\begin{proof}
Let $x\in \mathbb{H}$ with $||x||^\prime=r< r_0.$ Then  for $n>N_0$ and by using  \eqref{Bound4}:
\begin{eqnarray*}
N (\sum_{\stackrel{\nu=(n_1,n_2,n_3)}{n_1+n_2+n_3=n}}P_\nu a_\nu ) &\leq&\mkern-18mu\mkern-18mu \sum_{\stackrel{\nu=(n_1,n_2,n_3)}{n_1+n_2+n_3=n}} \frac{(||x||^{\prime })^n}{n!}||a_\nu||\binom {n}{ n_1,n_2,n_3}\\
&\leq&\sum_{\stackrel{\nu=(n_1,n_2,n_3)}{n_1+n_2+n_3=n}} \frac{r^n}{ n!} (\frac{M}{r_0^n}) \binom{n}{n_1,n_2,n_3}\\
&\leq& M\frac{3^n}{n!}.
\end{eqnarray*}

Which give us \[N(f)(x)\leq N(\sum_0^{N_0} \sum_{\stackrel{\nu=(n_1,n_2,n_3)}{n_1+n_2+n_3=n}} P_\nu a_\nu)(x)+Me^3.\]
\end{proof}

 The next lemma is well-known, we include it because our norm $\|\cdot\|^\prime$ simplifies the proof.

\begin{lemma}\label{All}
If \[\limsup_{n\rightarrow \infty} (\max_{||\nu||=n}\|a_\nu/n!\|)^{\frac{1}{n}}=0\] then  \[f(x)=\sum_0^\infty \sum_{\stackrel{\nu=(n_1,n_2,n_3)}{n_1+n_2+n_3=n}} P_\nu a_\nu,\]
 converges compactly for all $\mathbb{H}$.
\end{lemma}
\begin{proof}

Let $x\in \mathbb{H}-\{0\}.$  There is $N_0\in \mathbb{N}$ such that for all $n>N_0$ \[(\max_{||\nu||=n} ||a_\nu/n!||)^{\frac{1}{n}}\leq\frac{1}{6||x||^\prime}.\]  
Then using \ref{Bound4} we have that for all $n>N_0:$
\begin{eqnarray*}
\sum_{\stackrel{\nu=(n_1,n_2,n_3)}{n_1+n_2+n_3=n}}N (P_\nu a_\nu ) &\leq&\mkern-18mu\mkern-18mu \sum_{n_1+n_2+n_3=n} \frac{||x||^{\prime n}}{(6||x||^\prime)^n}\binom {n}{ n_1,n_2,n_3}\\
&=& \frac{1}{6^n }\sum_{n_1+n_2+n_3=n} \binom{n}{n_1,n_2,n_3}\\
&=& 1/2^n.
\end{eqnarray*}
We conclude that \[N(f)(x)\leq N(\sum_0^{N_0} \sum_{\stackrel{\nu=(n_1,n_2,n_3)}{n_1+n_2+n_3=n}} P_\nu a_\nu)(x)+2.\]
Compactly convergence follows from convergence of $N(f)$ accor\-ding to Lemma \ref{Aux}.
\end{proof}

The following theorem is useful in cases where there is a variation in the magnitude of the coefficients of the homogeneous components of the Taylor series or several of them have absolute values smaller than 1.

\begin{theorem}
(Quaternionic Cauchy-Hadamard Theorem) \label{Final}

Let \[\sigma=\limsup_n(\sum_{\stackrel{\nu=(n_1,n_2,n_3)}{n_1+n_2+n_3=n}} \binom{n}{n_1,n_2,n_3} \|a_{\nu}\|/n!)^{1/n}\] then  \[f(x)=\sum_0^\infty \sum_{\stackrel{\nu=(n_1,n_2,n_3)}{n_1+n_2+n_3=n}} P_\nu a_\nu,\] converges compactly for all $h\in B^\prime(0,\frac{1}{\sigma})$.
\end{theorem}\begin{proof}

The case  $\sigma = 0$ is follows from  Lemma (\ref{All}). Assuming $\sigma\neq 0,$ let  $x\in B^\prime(0,\frac{1}{\sigma})$ and let $\theta=\sqrt{||x||^\prime \sigma}<1,$ then \[\frac{\theta}{||x||^\prime}=\frac{\sigma}{\theta}>\sigma,\] we conclude that there is $N_0\in \mathbb{N}$ such that for all $n>N_0:$ \[(\sum_{\stackrel{\nu=(n_1,n_2,n_3)}{n_1+n_2+n_3=n}} \binom{n}{n_1,n_2,n_3} \|a_{\nu}\|/n!)^{\frac{1}{n}}\leq \frac{\theta}{||x||^\prime}.\]

 Then for  $n>N_0$ we can substitute on \eqref{Bound4}: 
\begin{eqnarray*}
 N(\sum_{\stackrel{\nu=(n_1,n_2,n_3)}{n_1+n_2+n_3=n}} P_{\nu}a_\nu )(x)&\leq& \frac{(||x||^{\prime })^{n}}{n!}\sum_{\stackrel{\nu=(n_1,n_2,n_3)}{n_1+n_2+n_3=n}}   \binom{n}{n_1,n_2,n_3} \|a_\nu\|\\
&\leq& ||x||^{\prime n}\frac{\theta^n}{||x||^{\prime n}}\\
&=& \theta^{n}.
\end{eqnarray*}
We conclude that \[N(f)(x)\leq N(\sum_0^{N_0} \sum_{\stackrel{\nu=(n_1,n_2,n_3)}{n_1+n_2+n_3=n}} P_\nu a_\nu)(x)+1/(1-\theta).\]
\end{proof}

 In the next theorem we compute the radius by considering the maximum absolute value of the coefficients multiplied by the number of non zero coefficients.

We introduce \[\tau=\limsup_n( \sum_{\stackrel{\nu=(n_1,n_2,n_3)}{|\nu|=n,  a_\nu \neq0}}  \binom{n}{n_1, n_2, n_3})^{1/n},\] 

For polynomials we compute $\tau=0$.  Any holomorphic function is an example of a series with $\tau=1$. The hyperholomorphic function (\ref{Sample}) has $\tau=3$, which is  the maximum value of $\tau$.
\begin{theorem}
(Weak Quaternionic Cauchy-Hadamard Theorem) \label{Cfinite} \label{CH}
Let \[\rho =\limsup_{n\rightarrow \infty} (\max_{||\nu||=n}||a_\nu/n!||)^{\frac{1}{n}}, 0\leq\rho<\infty\] and \[\tau=\limsup_n( \sum_{\stackrel{\nu=(n_1,n_2,n_3)}{|\nu|=n,  a_\nu \neq0}}  \binom{n}{n_1, n_2, n_3})^{1/n}\] then  \[f(x)=\sum_0^\infty \sum_{\stackrel{\nu=(n_1,n_2,n_3)}{n_1+n_2+n_3=n}} P_\nu a_\nu,\] converges compactly for all $h\in B^\prime(0,\frac{1}{\tau \rho})$.
\end{theorem}
\begin{proof} 
Lemma (\ref{All}) considers the case $\rho=0$. 
Let  $x\in B^\prime(0,\frac{1}{\tau\rho})$ and let $\theta=\sqrt{||x||^\prime \tau\rho}<1,$ then \[\frac{\theta}{\tau||x||^\prime}=\frac{\rho}{\theta}>\rho,\] we conclude that there is $N_0\in \mathbb{N}$ such that for all $n>N_0$ \[||a_\nu/n!||^{\frac{1}{n}}\leq \frac{\theta}{\tau||x||^\prime},\,\, ||\nu||=n.\]
Since \[\frac{\tau}{\theta^{1/2}}>\tau,\]
we can find  $M_0>0$ so that for all $M>M_0:$ \[ \sum_{\stackrel{\nu=(n_1,n_2,n_3)}{|\nu|=n,  a_\nu \neq0}}  \binom{n}{n_1, n_2, n_3}<(\frac{\tau}{\theta^{1/2}})^M.\]

 Then  for  $n>\max\{N_0, M_0\}:$
\begin{eqnarray*}
  N(\sum_{\stackrel{\nu=(n_1,n_2,n_3)}{n_1+n_2+n_3=n}} P_{\nu}a_\nu )(x)&\leq&(||x||^{\prime })^{n} \sum_{\stackrel{\nu=(n_1,n_2,n_3)}{n_1+n_2+n_3=n}}   \binom{n}{n_1,n_2,n_3}\frac{\|a_\nu\|}{n!} \\
  &\leq& \sum_{a_\nu\neq 0, |\nu|=n}\binom{n}{n_1,n_2,n_3} ||x||^{\prime n}{}\frac{\theta^n}{||x||^{\prime n}\tau^n}\\
&=& (\tau/\theta^{1/2})^n(\theta/\tau)^n\\
&=& \theta^{n/2}.
\end{eqnarray*}
We conclude that \[N(f)(x)\leq N(\sum_0^{N_0} \sum_{\stackrel{\nu=(n_1,n_2,n_3)}{n_1+n_2+n_3=n}} P_\nu a_\nu)(x)+1/(1-\sqrt{\theta}).\]
\end{proof}

\section{Comparison of radius of convergence.}
\label{Comparison}

In the previous section we showed that our radius  does describe a region when the Taylor series expansion converges. In this section we will compare out methods with the ones on the literature.

Using [\cite{Malonek1990}, page 9], we need to compute  \[ \rho_1 = \lim_{n \rightarrow \infty} ((\max_{\|\nu\|=n}{\ |a_\nu\|})^{1/n})^{-1}\] to guarantee that $f$  converges on $\{x\in\mathbb{H}|\,\,\|x\|<\rho_1\}$.

For example given the function

\begin{eqnarray}\label{Sample}\sum_0^\infty\sum_{|\nu|=n} n! P_\nu&=&\sum_0^\infty\sum_{|\nu|=n} \sum_{(i_1,\cdots,i_n)\in A_\nu} \zeta_{i_1}\cdots \zeta_{i_n} \end{eqnarray}

\[ \rho_1 = \lim_{n \rightarrow \infty} ((\max_{\|\nu\|=n}{1})^{1/n})^{-1}\] 
Let $q=1/3(i+j+k)$, $\|q\|=1/\sqrt{3}$ but the series diverges on this point:

\begin{eqnarray*}\sum_0^\infty\sum_{|\nu|=n} n! P_\nu(q)&=&\sum_0^\infty\sum_{|\nu|=n} \sum_{(i_1,\cdots,i_n)\in A_\nu} \zeta_{i_1}\cdots \zeta_{i_n}(q)\\
&=&\sum_0^\infty\sum_{|\nu|=n} \sum_{(i_1,\cdots,i_n)\in A_\nu} 1/3^n\\
&=&\sum_0^\infty\sum_{|\nu|=n, \nu=( n_1, n_2, n_3)} \binom{n}{ n_1, n_2, n_3} 1/3^n\\
&=&\sum_0^\infty (3/3)^n
\end{eqnarray*} 

The reason is that we also need to count how many terms are non zero. This motivated the definition of  $\tau$ in Theorem \ref{CH}. We compute  $\rho = 1/3$ for this series.

Following [\cite{Holo2014}, page 168], we use the expression :

\[\sum_0^\infty \sum_{\stackrel{(n_0n_1,n_2,n_3)}{n_0+n_1+n_2+n_3=n}} |x_0^{j_0}x_1^{j_1}x_2^{j_2}x_3^{j_3}|\|a_{(n_0n_1,n_2,n_3)}\|
\]
to compute  the formula 
\begin{eqnarray}\label{r2}
\rho_2 &=& \lim_{n \rightarrow \infty} (\sum_{\|(v_0,v_1,v_2,v_3)\|=n}{\|a_{(v_0,v_1,v_2,v_3)}}\|^{1/n})^{-1}
\end{eqnarray}

We  consider a complex analytic function and we decompose it in the real variables, to  obtain a series $ f(z)=\sum z^n a_n= \sum_n \sum \binom{n}{s}x^{n-s}y^{s}(-i)^sa_n$, then using  
\begin{eqnarray}\lim_{n \rightarrow \infty} ((\sum_{s=0}^n\|\binom{n}{s}(-i)^s a_n\|)^{1/n})^{-1}&=&\lim_{n \rightarrow \infty} ((2^n\| a_n\|)^{1/n})^{-1}\\
&=&1/2\lim_{n \rightarrow \infty} ((\| a_n\|)^{1/n})^{-1}.
\end{eqnarray} 

 The inequalities $|x|<\|z\|, |y|<\|z\|$ require to consider all binomial terms on $(x+iy)^n$ and this is the source of the $2^n$ coefficient.
Since the norm is multiplicative we can avoid the coefficient 1/2 by just considering $\|a_nz^n\|=\|a_n\|\|z\|^n$.

In the quaternionic case we the variable $x_0+ix_1+jx_2+kx_3$ is not  hyperholomorphic. Instead we have the decomposition on the $P_\nu$  polynomials. When adding the coefficients of the real variables instead of the coefficients of the Fueter variables, the coefficients $a_\nu$ always contribute $2^n$ times as
\[\zeta_{i_1}\cdots \zeta_{i_n}=  (x_{i_1}-i_0 x_{i_0})\cdots (x_{i_n}-i_nx_{i_0})=\sum \prod c_I x_{I_0}\cdots x_{I_n}\] contains $2^n$ monomials of real variables.

Similarly to the complex case, a radius that depends on a decomposition on real variables is not the maximal radius of convergence due to this $2^n$ coefficient. 

In our norm we consider

\begin{eqnarray*}
\|\zeta_{i_1}\cdots \zeta_{i_n}\|&=&\|\zeta_{i_1}\|\cdots \|\zeta_{i_n}\|.
\end{eqnarray*}

Since we stick to the Fueter variables our radius only considers once copy of $a_\nu$ on the calculation.

\label{Basis}

Our norm returns a radius $\rho$ so that a region of the form $B^\prime(\rho)=\{x|\|x\|^\prime<\rho\}$ is contained in the region of convergence.  Different basis can be used to obtain a Taylor expansion of a hyperholomorphic function, see for example \cite{Alpay2011}. 
  In terms of Theorem \ref{CH},  if we use a different basis we will vary the parameters $\tau$ and $ \rho$. This happens because in the new basis we will obtain a rotated $\|\cdot\|^\prime$ ball and that region should be contained in the maximal region of convergence. The radius $\rho$ will decrease or increase accordingly. 
  
  For example, let $ \sum(x_i-kx_j)^n=\sum(\zeta_i-\zeta_j k)^n.$ This series converges on a `tubular' region of those quaternions with $\|x_i+kx_j\|<1$. Note that $B^\prime(0,1)$ contains points outside of $\|x_i+kx_j\|<1$ as $.9i+.9j$. If we write $\sum(x_i+kx_j)^n$  in the Fueter basis then the number of coefficients $a_{\vec{\nu}}$ increases and so does $\tau$, allowing $B^\prime(0,1/\rho\tau)$ to be contained in $\|x_i+kx_j\|<1.$ 
  
 The expression $ (\zeta_i+\zeta_j k)^n $ contains terms of the form $\zeta_ik\zeta_j, $ which are not hyperholomorphic. Thus, by only expanding $\sum (\zeta_i+\zeta_j k)^n $ will not lead to the hyperholomorphic expression. For example the correct hyperholomorphic expression of $(\zeta_i+\zeta_j k)^2$ is $ \zeta_i^2-\zeta_j^2+(\zeta_i\zeta_j+\zeta_j\zeta_i)k$.

\subsection{Examples of domains}

It is important to work with open domains.  As any holomorphic function induces a hyperholomorphic function;  the series $\sum\zeta_1^nn^n+2$ converges only on the plane $j\mathbb{R}+k\mathbb{R},$ where it has  the constant value 2.

Consider $\sum \zeta_1^n a_n +$ $\sum \zeta_2^n b_n+$ $\sum  \zeta_3^n c_n$  with
 $s_1 =\limsup_{k\rightarrow \infty} (||a_k||)^{\frac{1}{k}}, $ $s_2 =\limsup_{k\rightarrow \infty} (||b_k||)^{\frac{1}{k}}, $ $s_3 =\limsup_{k\rightarrow \infty} (||c_k||)^{\frac{1}{k}};$  
then Theorem \ref{CH} guarantees that $f(x)=\sum \zeta_1^n a_n + \zeta_2^n b_n+  \zeta_3^n c_n$ convergences on $B^\prime(0,\frac{1}{s}), s=\max\{s_1,s_2,s_3\}$. On the other hand, Theorem \ref{Aux} give us a bigger domain of convergence \[\{ ||x_0+ix_1||<\frac{1}{s_1},||x_0+jx_2||<\frac{1}{s_2},||x_0+kx_3||<\frac{1}{s_3}\}.\]

Here is an example when the domain of convergence of the function is exactly a poly-cylinder $f(x)=\sum \zeta_1^{2^n}  + \zeta_2^{2^n} +  \zeta_3^{2^n} $.
And here is an example when the radius of convergence is not rational: let's consider $\nu_k=(4k, k,  k)$, then using Stirling formula we obtain 
\[\lim_{k\rightarrow \infty}{\binom{6k}{  4k, k,  k}}^{\frac{1}{6k}}=\frac{3}{2^{\frac{1}{3}}},\]
and so  
 the series $\sum_{ k=0}^{ \infty} P_{\nu_k}$ has radius of convergence  $2^{1/3}/3$.

\section{Octonions}\label{O}
 The non commutativity of Quaternions set the shape of the Taylor expansions.   We used the fact that the norm $\|\cdot\|$ is multiplicative to bound the Taylor expansion in terms of the norm $\|\cdot\|^\prime$, which is based on the Fueter's basis. We never worried about commutativity since we worked all the time with the norms of the coefficients and the terms $\|\zeta_s\|^\prime$. 

The only real finite normed division algebras are $R, C, H, O$. There is a Taylor expansion of functions in terms of the Fueter basis for Octonion hyperholomorphic functions. Since we have a multiplicative norm in the Octonions we can carry out our proofs without having to worry about non associativity, the Taylor expansion of the functions is shaped by the non associative product, but then we will only work with the norms of the coefficients.

We consider $O$ as the real vector space generated by $e_0, e_1,\cdots, e_7$, with multiplication rule \cite{Octo3}

\[
{ e_{i}e_{j}={\begin{cases}e_{j},&{\text{if }}i=0\\e_{i},&{\text{if }}j=0\\-\delta _{ij}e_{0}+\varepsilon _{ijk}e_{k},&{\text{otherwise}}\end{cases}}}\]
where ${ \delta _{ij}} \delta _{ij}$ is the Kronecker delta and ${ \varepsilon _{ijk}} \varepsilon _{ijk}$ is a completely antisymmetric tensor with value +1 when ijk = 123, 145, 176, 246, 257, 347, 365.

\begin{definition}\cite{Octo2} A function $h:O\rightarrow O$ is said to be (left)  $O$ -analytic in a neighborhood $V$ of the origin, if $h$ is real differentiable on $V$ and if
$D_Of=0$ when $D_O$ is the Octonion Cauchy-Riemann-Fueter operator $D_O= \sum_0^7 e_i \frac{\partial}{\partial x_i}.$
\end{definition}

In \cite{Octo2} they prove 
 the following Taylor expansion for $O$-analytic functions:

\begin{eqnarray*}
h(x)&=&\sum_{n=0}^\infty \sum_{\stackrel{\nu=(n_1,n_2,n_3,n_4,n_5,n_6,n_7)}{n_1+n_2+n_3+n_4+n_5+n_6+n_7=n}} V_\nu c_\nu,\end{eqnarray*}
 where 
\[V_\nu=\frac{1}{n!}\sum_{(i_1,\cdots,i_n)} ((\cdots(((\zeta_{i_1}\zeta_{i_2})\zeta_{i_2})\cdots )\zeta_{i_n}),\] 
where $\zeta_i=(x_i-e_ix_0)$ and the sum is over  is the set of all possible words with  $n_1$ letters $1$, $ n_2$ letters $2$, etc.  Note that an order of multiplication for the $\zeta$ functions is fixed.

\begin{definition}
We denote by $||\,\cdot\,||_O^\prime:\mathbb{V}\rightarrow \mathbb{R}$ the norm 
\[||v_0+\sum e_i v_i||_O^\prime=\max_{i\in \{1,\cdots,7\}}\{||v_i-e_iv_0||\}.\]\end{definition}

We compute
\begin{eqnarray}
\hbox{ volume }(B^\prime(0,r))&=&2^8\int^r_{t=0}(r^2-t^2)^{7/2}dt\\
&=&r^8*35\pi\\
&=&24*35/\pi^3
\hbox{ volume }(B(0,r))\\
\end{eqnarray}

It turns out that we can prove the equivalent to the main theorems on this paper with the same techniques as the previous sections by replacing occurrences of $\binom{n}{n_1,n_2,n_3}$ with $\binom{n}{n_1,n_2,n_3,n_4,n_5,n_6,n_7}$.

\begin{theorem} (Octonion Abel Theorem)\label{AbelO} Suppose that there are constants $r_0, M\in \mathbb{R}, N_0\in \mathbb{R},$ such that for all $n>N_0$ and multi indexes $\nu$ with $||\nu||=n;$ we have the bound $||c_\nu||r^n_0\leq M$. Under this hypothesis  the series \begin{eqnarray}\label{FuncO}\label{FunctionO}
h(x)&=&\sum_0^\infty \sum_{\stackrel{\nu=(n_2,n_3)}{n_2+n_3=n}} V_\nu c_\nu,\end{eqnarray}  converges compactly on $\{x|\,\,\,||x||_O^\prime<r_0\}.$
\end{theorem}

\begin{theorem}
(Octonion Cauchy-Hadamard Theorem) 

Let \[\sigma=\limsup_n(\sum_
{\stackrel{\nu=(n_1,n_2,n_3,n_4,n_5,n_6,n_7)}{\sum n_i=n}}
 \binom{n}{\nu}\|\frac{c_{\nu}}{n!}\|)^{1/n}\] then  \eqref{FuncO} converges compactly for all $h\in \{x|\,\,\,||x||_1^\prime<\frac{1}{\sigma}\}$.
\end{theorem}

\begin{theorem} ( Weaker Octonion Cauchy-Hadamard Theorem )

  Let \[\rho =\limsup_{k\rightarrow \infty} (\max_{\|\nu||=k}\|c_\nu/n!\|)^{\frac{1}{k}}
\] and \[\tau=\limsup_n(\sum_
{\stackrel{\nu=(n_1,n_2,n_3,n_4,n_5,n_6,n_7)}{\sum n_i=n|c_{\nu}\neq 0}} \binom{n}{\nu} )^{1/n},\] then (\ref{FuncO}) 
converges compactly on $\{x|\,\,\,||x||_O^\prime<\frac{1}{{\tau \rho}}\}$.
\end{theorem}

This all depends on having the function written in terms of the basis $\{\zeta_i\}_i$, the election of a different basis  of linear O-analytic polynomials will change the radius of convergence to allow a maximal $S$-bicylinder of convergence as in \ref{Basis}.

\section{Moisil–Th\'eodoresco Basis}\label{MT}
We consider $V\subset \mathbb{H}$ as the vector space generated by $i,j,k$, elements are written as  $v=iv_1+jv_2+kv_3$. We are interested in functions $g:\mathbb{V}\rightarrow \mathbb{H}$ that are real differentiable on $V$ and such that  $D_{MT}g=0$ when $D_{MT}$ is the Moisil–Th\'eodoresco operator \cite{quater}:
\[D_{MT}=i\frac{\partial }{\partial x_1}+j\frac{\partial }{\partial x_2}+k\frac{\partial }{\partial x_3}.\]

When the domain is $\mathbb{H}$, the first coordinate has different algebraic pro\-perty $1^2=1$ than the other coordinates $i^2=j^2=k^2=-1$ and so, to find the power series expansion of a function $f:\mathbb{H}\rightarrow \mathbb{H}$ it is a common  practice to use Fueter's basis, where the real variable has a different role as the other variables. Although any other variable can be selected to generate the corresponding homogeneous polynomials. Now we are working with $V\sim R^3$, where the three variables have the same algebraic properties, so our methods are not motivated by anti-symmetries anymore.

To apply our results we consider that we are working with a hyperholomorpic function that independent of the real coordinate. We consider  the basis: $$\{\xi_2(v)=v_2-\frac{v_1}{i}j, \xi_3(v)=v_3-\frac{v_1}{i}k \}.$$

Below we show that real analytic functions $g$ that satisfy $D_{MT}(g)=0$ on a open neighborhood of the origin can be expanded locally as: 
\begin{eqnarray}\label{FuncMT}
g(x)&=&\sum_{n=0}^\infty \sum^n_{\stackrel{k=0}{\nu=(k,n-k)}} S_\nu b_\nu,\end{eqnarray}  Where for every $\nu$, $n!S_\nu$ is a   polynomial obtained by adding all possible products of $k$ functions $\xi_2$ and $n-k$ functions $\xi_3$, and $b_\nu\in \mathbb{H}$.

Given an infinite differentiable function $g:\mathbb{V}\rightarrow \mathbb{H}$, and a quaternion $h=h_0+ih_1+jh_2+kh_3,$ we formally consider the  series:
$$T(g)(h)=\sum_{n=0}^\infty \frac{1}{n!}(h_1\frac{\partial}{\partial x_1}+h_2\frac{\partial}{\partial x_2}+h_3\frac{\partial}{\partial x_3})^nf|_{(0)}.$$

Since $D_{MT}(g)=0$:
 $$\frac{\partial}{\partial x_1}f=(-i^{-1}j\frac{\partial }{\partial x_2}-i^{-1}k\frac{\partial }{\partial x_3})f$$  and we can rewrite the $n$-derivative as:
$$ \frac{1}{n!}\left((h_2-i^{-1}jh_1)\frac{\partial}{\partial x_2}+(h_3-i^{-1}kh_1)\frac{\partial}{\partial x_3}\right)^nf|_{(0)}.$$

Given $\nu=(s,n-s)$ let $B_\nu $ be the set of all possible vectors with $s$ numbers $2$ and $n-s$ numbers $3$. Then

\begin{eqnarray*}   \lefteqn{\frac{1}{n!}\left(\xi_2(h)\frac{\partial}{\partial x_2}+\xi_3(h)\frac{\partial}{\partial x_3}\right)^n\!\!\!f|_{(0)}\!\!\!=}\\
&=&\!\!\!\frac{1}{n!}\sum_{|\nu|=n}\!\sum_{\,\vec{\nu}=(i_1,\cdots,i_n)\in B_\nu}\!\!\!\!\!\! \!\!\xi_{i_1}\cdots \xi_{i_n}b_{{\nu}}\\
&=& \sum S_\nu b_\nu.
\end{eqnarray*}

\begin{definition}
We denote by $||\,\circ\,||_1^\prime:\mathbb{V}\rightarrow \mathbb{R}$ the norm 
\[||iv_1+jv_2+kv_3||_1^\prime=\max\{||v_2-i^{-1}jv_1||,||v_3-i^{-1}kv_1||\}.\]\end{definition} 
The balls determined by $||\,\circ\,||_1^\prime$ are still bigger than the euclidean ones, in fact they are bicylinders as in figure
\ref{Fig:mesh1} \footnote{ Image made by Ag2gaeh - Own work, CC BY-SA 4.0, https://commons.wikimedia.org/w/index.php?curid=63519897 .}.
\begin{figure}[ht]
    \centering
    \includegraphics[width=0.25\textwidth]{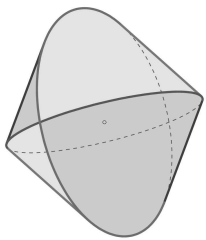}
    \caption{Bicylinder}
    \label{Fig:mesh1}
\end{figure}

It turns out that we can prove the equivalent to the main theorems on this paper with the same techniques as the previous sections by replacing occurrences of $\binom{n}{n_1,n_2,n_3}$ with $\binom{n}{n_1}$.

\begin{theorem} (Abel Theorem)\label{AbelMT} Suppose that there are constants $r_0, M\in \mathbb{R}, N_0\in \mathbb{R},$ such that for all $n>N_0$ and multi indexes $\nu$ with $||\nu||=n;$ we have the bound $||b_\nu||r^n_0\leq M$. Under this hypothesis  the series \begin{eqnarray}\label{FuncMTS}\label{FunctionMT}
g(x)&=&\sum_0^\infty \sum_{\stackrel{\nu=(s,t)}{s+t=n}} S_\nu b_\nu,\end{eqnarray}  converges compactly on $\{x|\,\,\,||x||_1^\prime<r_0\}.$
\end{theorem}

\begin{theorem}
(Cauchy-Hadamard Theorem) 

Let \[\sigma=\limsup_n(\sum_{s=0}^n \binom{n}{s}\|b_{(s,n-s)}/n!\|)^{1/n}\] then  \eqref{FuncMTS} converges compactly for all $h\in \{x|\,\,\,||x||_1^\prime<\frac{1}{\sigma}\}$.
\end{theorem}

\begin{theorem} ( Weaker Cauchy-Hadamard Theorem )

  Let \[\rho =\limsup_{k\rightarrow \infty} (\max_{\|\nu||=k}\|b_\nu/n!\|)^{\frac{1}{k}}
\] and \[\tau=\limsup_n(\sum_{0\leq s\leq n| b_{s, n-s}\neq 0} \binom{n}{s} )^{1/n},\] then (\ref{FuncMTS}) 
converges compactly on $\{x|\,\,\,||x||_1^\prime<\frac{1}{{\tau \rho}}\}$.
\end{theorem}

This all depends on having the function written in terms of the basis $\{\xi_2(v)=v_2-\frac{v_1}{i}j, \xi_3(v)=v_3-\frac{v_1}{i}k \}$, the election of a different basis $S=\{S_a, S_b\}$ of linear polynomials that satisfy $D_{MT}S_t=0, t=a,b,$ will change the radius of convergence to allow a maximal $S$-bicylinder of convergence as in \eqref{Basis}.

\section{Future work}
The norm $\|\cdot\|^\prime$  may give better estimations of error approximation by Taylor series. 

Is $\rho$  the radius of the  maximal $\|\cdot\|^\prime$ ball contained on the corresponding domain of convergence?
Closely related to the previous question, note that due to the combinatorial definition of $P_\nu$, and the non commutativity of Quaternions, there may be a way to bound the possible values of $P_\nu$ based only on the norm of the input and the coefficients $\nu$. Perhaps this will avoid the coefficient $3^n$ in the Quaternionic case or the coefficient $7^n$ in the Octonion case.

It would be useful to have an explicit relationship between the region calculated using equation \eqref{r2} and our region.

\section{Acknowledgements} 
The author would like to thank  Dr. Michael  Shapiro for many useful discu\-ssions and valuable comments, as well as
the support from  Dr. Jose Mendoza-Cortez. I am responsible for mistakes still present in this version. The author would like to thank the anonymous reviewer for the constructive feedback.

The author thank CONACyT for providing the 
undergraduate Fellowship PIFI  project cgip 20070253.

\bibliography{main.bib}{}

\begin{thebibliography}{1}

\bibitem{Alpay2011}
Daniel Alpay, Flor de~Mar{\'{i}}a Correa-Romero, Mar{\'{i}}a~Elena
  Luna-Elizarrar{\'{a}}s, and Michael Shapiro.
\newblock {On the Structure of the Taylor Series in Clifford and Quaternionic
  Analysis}.
\newblock {\em Integral Equations and Operator Theory}, 71(3):311--326, 2011.

\bibitem{quater}
Juan Bory~Reyes and Michael Shapiro.
\newblock Clifford analysis versus its quaternionic counterparts.
\newblock {\em Mathematical Methods in the Applied Sciences}, 33(9):1089--1101,
  2010.

\bibitem{Fueter}
Rud. Fueter.
\newblock {Die Funktionentheorie der Differentialgleichungen ...u = 0 und ....u
  = 0 mit vier reellen Variablen}.
\newblock {\em Commentarii mathematici Helvetici}, 7:307--330, 1934.

\bibitem{Holo2014}
Klaus G{\"{u}}rlebeck, Klaus Habetha, and Wolfgang Spr{\"{o}}{\ss}ig.
\newblock {\em {Holomorphic functions in the plane and n-dimensional space}}.
\newblock Springer Science {\&} Business Media, ilustrated edition, 2008.

\bibitem{Octo2}
J.~Liao and J.~Wang.
\newblock The analyticity for the product of analytic functions on octonions
  and its applications.
\newblock {\em Advances in Pure Mathematics}, 7(12):692--705, 2017.

\bibitem{Malonek1990}
H~Malonek.
\newblock {Power series representation for monogenic functions in {$R^{n+1}$}
  based on a permutational product}.
\newblock {\em Complex variables}, 15(July):181--191, 1990.

\bibitem{Octo3}
Lev Vasilevitch Sabinin; Larissa Sbitneva; I.~P. Shestakov.
\newblock Octonion algebra and its regular bimodule representation.
\newblock {\em Non-associative algebra and its applications,}, 17(2):235, 2006.

\end{thebibliography}
\bibliographystyle{plain}

\end{document}